\documentclass[11pt]{amsart}
\usepackage[utf8]{inputenc}
\usepackage{amsmath}
\usepackage{amsfonts}
\usepackage{amssymb}
\usepackage{mathrsfs}
\usepackage{graphicx}
\usepackage{bm}
\usepackage{mathtools}
\usepackage{bbm}
\usepackage{stmaryrd}
\usepackage{graphicx}
\usepackage{float}
\usepackage{yhmath}

\usepackage[left=3.2cm,right=3.2cm, top=3.1cm,bottom=3.1cm]{geometry}
\makeatletter
\renewcommand{\@secnumfont}{\bfseries}
\makeatother

\makeatletter

\renewcommand\subsubsection{\@secnumfont}{\bfseries}%
\renewcommand\subsubsection{\@startsection{subsubsection}{3}
  \z@{.5\linespacing\@plus.7\linespacing}{-.5em}%
  {\normalfont\bfseries}}
  
  \makeatother

\usepackage[colorlinks=true, linkcolor=cyan!20!blue, citecolor=purple, urlcolor=purple]{hyperref}

\usepackage[capitalise]{cleveref}
\usepackage{longtable}

\usepackage{aliascnt}
\usepackage{xparse}

\let\oldnewtheorem\newtheorem
\RenewDocumentCommand{\newtheorem}{s m o m O{}}{%
  \IfBooleanTF{#1}% starred form
    {\oldnewtheorem{#2}{#4}}%
    {%
      \IfNoValueTF{#3}%
        {\oldnewtheorem{#2}{#4}[#5]}%
        {%
          \newaliascnt{#2}{#3}%
          \oldnewtheorem{#2}[#2]{#4}%
          \aliascntresetthe{#2}%
        }%
    }%
}
\usepackage{tikz-cd} 
\usepackage[all,cmtip]{xy}	
\xyoption{arrow}
\usepackage{dynkin-diagrams}

\usepackage{xcolor,etoolbox}
\usepackage{enumitem}
\usepackage{adjustbox}

\usepackage[textsize=footnotesize, textwidth=25mm, color=green!40]{todonotes}

\usepackage{array}
\usepackage{multirow}

\theoremstyle{plain}

\newtheorem{theorem}[subsubsection]{Theorem}
\newtheorem{thm}[subsubsection]{Theorem}

\Crefname{theorem}{theorem}{theorems}
\Crefname{theorem}{Theorem}{Theorems}

\newtheorem{prop}[theorem]{Proposition}
\newtheorem{cor}[theorem]{Corollary}

\theoremstyle{definition}
\newtheorem{defi}[theorem]{Definition}
\newcommand{\pp}{\medskip} 

\newtheorem{rmk}[theorem]{Remark}
\newtheorem{eg}[theorem]{Example}
\newtheorem{ex}[theorem]{Example}

\newcommand{\ZZ}{\mathbb{Z}}
\newcommand{\QQ}{\mathbb{Q}}

\newcommand{\NN}{\mathbb{N}}
\newcommand{\CC}{\mathbb{C}}
\newcommand{\PP}{\mathbb{P}}
\newcommand{\VV}{\mathbb{V}}

\newcommand{\DD}{\mathbb{D}}

\newcommand{\Aa}{\mathsf{A}}

\newcommand{\Cl}[1]{\CC(\!({#1})\!)}

\newcommand{\Cc}{\mathcal{C}}
\newcommand{\Pc}{\mathcal{P}}
\newcommand{\Hc}{\mathcal{H}}
\renewcommand{\Mc}{\mathcal{M}}
\newcommand{\Vc}{\mathcal{V}}

\newcommand{\Oc}{\mathscr{O}}
\newcommand{\Oo}{\mathcal{O}}
\newcommand{\Lc}{\mathcal{L}}

\newcommand{\Ak}{\mathfrak{A}}
\newcommand{\Ck}{\mathfrak{C}}
\newcommand{\Ic}{\mathcal{I}}
\newcommand{\Lk}{\mathfrak{L}}

\usepackage{oldgerm}

\newcommand{\Ac}{\mathcal{A}}
\newcommand{\Ls}{\mathsf{L}}

\newcommand{\Hm}{\mathrm{H}}
\newcommand{\Spec}{\mathrm{Spec}}
\newcommand{\End}{\mathrm{End}}
\newcommand{\Hom}{\mathrm{Hom}}

\newcommand{\Res}{\mathrm{Res}}
\newcommand{\vac}{\bm{1}}
\newcommand{\conf}{\bm{\omega}}
\newcommand{\id}{\mathrm{id}}

\newcommand{\g}{\mathfrak{g}}
\newcommand{\gh}{\widehat{\g}}
\newcommand{\Gm}{\mathbb{G}_m}
\newcommand{\Vir}{\mathrm{Vir}}
\newcommand{\cc}{\bm{c}}

\newcommand{\Ct}{\widetilde{C}}

\newcommand{\cbotimes}{\underset{\VV}{\otimes}}

\newcommand{\Vmod}{V\mathsf{\!-mod}}

\usepackage{scalerel}
\usepackage{stackengine}
\stackMath

\newcommand{\simeqaccent}[1]{%
  \stackengine{0ex}{#1}{%
    \smash{\ensurestackMath{%
      \stackengine{0.2ex}{\overline{\phantom{#1}}}{\widetilde{\phantom{#1}}}{O}{c}{F}{F}{L}%
    }}%
  }{O}{c}{F}{F}{L}%
}

\newcommand{\triangleaccent}[1]{%
  \stackengine{0ex}{#1}{%
    \smash{\ensurestackMath{%
      \stackengine{0.2ex}{\overline{\phantom{#1}}}{\widehat{\phantom{#1}}}{O}{c}{F}{F}{L}%
    }}%
  }{O}{c}{F}{F}{L}%
}

\newcommand{\bTMgn}{\simeqaccent{{\Mc}}_{g,n}}
\newcommand{\btMgn}{\triangleaccent{{\Mc}}_{g,n}}
\newcommand{\bMgn}{{\overline{\Mc}}_{g,n}}

\newcommand{\bM}{\overline{\Mc}}

\newcommand{\TMgn}{\widetilde{{\Mc}}_{g,n}}

\newcommand{\Mgn}{{{\Mc}}_{g,n}}

\newcommand{\TDelta}{\widetilde{\Delta}}

\makeatletter
\newcommand{\ostar}{\mathbin{\mathpalette\make@circled\star}}

\newcommand{\make@circled}[2]{%
  \ooalign{$\m@th#1\smallbigcirc{#1}$\cr\hidewidth$\m@th#1#2$\hidewidth\cr}%
}

\newcommand{\smallbigcirc}[1]{%
  \vcenter{\hbox{\scalebox{0.7}{$\m@th#1\bigcirc$}}}%
}
\makeatother

\newcommand{\Gc}{\mathcal{G}}
\newcommand{\Tc}{\mathcal{T}}

\newcommand{\Ec}{\mathcal{E}}

\newcommand{\Kc}{\mathcal{K}}

\newcommand{\Db}{\mathbb{D}}

\newcommand{\Ho}{\mathrm{H}}

\newcommand{\Lie}{\mathrm{Lie}}
\newcommand{\Pic}{\mathrm{Pic}}

\newcommand{\Bun}{\mathrm{Bun}}

\newcommand{\Gr}{\mathrm{Gr}}

\newcommand{\SL}{\mathrm{SL}}

\newcommand{\GL}{\mathrm{GL}}

\newcommand{\h}{\mathfrak{h}}

\usepackage[justification=centering]{caption}

\begin{document}

\title[Conformal blocks in algebraic geometry]{Conformal blocks in algebraic geometry}
\author[Chiara Damiolini]{Chiara Damiolini}
\address{Department of Mathematics, University of Texas at Austin, Austin, TX} 
\email{chiara.damiolini@austin.utexas.edu}
\thanks{This work is partially supported by NSF DMS 2401420}

\begin{abstract} These notes survey the theory of (twisted) conformal blocks from an algebro–geometric perspective and have two main goals. The first one is to summarize the construction of conformal blocks from vertex operator algebras, and to describe their fundamental properties---such as factorization and sewing---which imply that conformal blocks define vector bundles on moduli spaces of curves whose ranks and Chern classes can be computed explicitly. The second aim is to describe how line bundles on moduli of principal bundles over a curve---and their sections---can be understood via (twisted) conformal blocks for (twisted) affine Lie algebras.     
\end{abstract}

\maketitle

\section{Introduction}

\emph{Conformal blocks} originate in two–dimensional conformal field theory, where they arise as spaces of chiral correlation functions associated with a Riemann surface equipped with marked points and representation–theoretic data attached to them. Given their nature, conformal blocks naturally sit at the intersection of representation theory, algebraic geometry, and mathematical physics. In particular, they provide a powerful bridge between the representation theory of infinite–dimensional algebras---such as affine Lie algebras and vertex operator algebras---and the geometry of moduli spaces of curves and principal bundles.

\pp From a mathematical perspective, the space of conformal blocks is best thought of as a vector space associated to a \emph{decorated curve}: a smooth (or stable) projective curve together with a finite set of marked points, each labeled by a suitable representation. Depending on the context, these representations may come from integrable highest–weight representations of affine Lie algebras or from modules over a vertex operator algebra (VOA). Although the decorations are local in nature, conformal blocks depend globally on the curve. When the curve varies in families these spaces assemble into sheaves---and in favorable cases vector bundles---over moduli spaces of curves. Under appropriate assumptions, these sheaves come equipped with rich additional structures. For instance, they admit a projectively flat connection and they satisfy certain compatibilities under pullback along tautological maps between moduli of curves. As an application, one can compute their ranks and Chern classes explicitly. These properties also have interesting consequences for the structure of the category of representations: via conformal blocks, it inherits a monoidal structure.

\pp Among the driving forces behind the mathematical development of the  subject were Verlinde's conjecture \cite{verlinde:1988} on a formula  for the dimensions of spaces of conformal blocks in terms of modular  data, and the work of Moore and Seiberg \cite{moore.seiberg:1989},  who identified the consistency conditions of rational conformal field  theories with the axioms of what mathematicians now call a \emph{modular functor}.  
Building on earlier work in genus zero, the first systematic algebro-geometric construction of conformal blocks on higher-genus curves was given by Tsuchiya, Ueno, and Yamada \cite{tuy}. Their approach defines conformal blocks as duals of spaces of coinvariants, obtained by quotienting tensor products of representations of affine Lie algebras by a Lie algebra encoding global functions on the punctured curve. {Vertex operator algebras} provide a far–reaching generalization of this picture. They simultaneously extend the notions of Lie algebras and associative algebras, while retaining enough structure to encode the operator product expansion fundamental to conformal field theory. Conformal blocks associated with VOA modules---constructed in \cite{NT,bzf,frenkelSzczesny:twusted,DGT1, DGT2}---recover those arising from affine Lie algebras as a special case, but also encompass a much broader range of examples, including those arising from lattice VOAs, Virasoro VOAs, and logarithmic theories.  In this setting, conformal blocks are again defined as dual of coinvariants, now with respect to a global Lie algebra attached to the VOA, which is much harder to control. Given the increased complexity, many of the key geometric properties mentioned above hold only under additional assumptions or in a modified form, and many questions concerning their behavior remain open.

\pp A central geometric aspect of conformal blocks from affine Lie algebras is their intimate relationship with the moduli stack $\Bun_G$ of principal $G$-bundles on a curve, where $G$ is a simple, simply connected complex algebraic group. For affine Lie algebras at positive integral level, the associated space of conformal blocks on a smooth projective curve can be canonically identified with spaces of global sections of certain line bundles on $\Bun_G$ \cite{beauville.laszlo:1994:conformal,beauville:96:verlinde,laszlo.sorger:1997,sorger:96:verlinde, sorger:99:moduli}. This perspective provides a geometric realization of the Verlinde formula, which computes the dimension of conformal blocks in terms of representation–theoretic data. This description is also extended to conformal blocks arising from \textit{twisted affine Lie algebras}---called \textit{twisted conformal blocks}---recovering sections of line bundles on the moduli stack $\Bun_\Gc$ of principal $\Gc$-bundles for a \textit{parahoric Bruhat--Tits} group $\Gc$ \cite{damiolini:2020:conformal,hongkumar:2023,damiolini.hong.gao:2025}. It is unknown whether, starting with representations of VOAs, the associated conformal blocks have a cohomological interpretation of a similar nature.

\pp The goal of this note is to provide a coherent narrative concerning the bulk of the theory of conformal blocks from an algebro-geometric perspective. Given my expertise, I focus on two main topics: properties of conformal blocks from VOAs, and the relation between twisted conformal blocks and parahoric Bruhat--Tits groups. When possible, I included references to research directions that I could not expand on.

\subsection*{Summary of the paper} In \cref{sec:CB} we introduce the definitions of conformal blocks as duals of spaces of coinvariants. In \cref{sec:CBfromLieAlgebras} we start with the case of affine Lie algebras, while in \cref{sec:CBfromVOA} we explain their generalization to VOAs. Twisted conformal blocks are discussed in \cref{sec:twCB}.

\Cref{sec:properties} is devoted to the geometric properties of conformal blocks. In particular, in \cref{sec:cor} we collect several consequences of the properties described in \crefrange{sec:sheaf}{sec:projconn}, including local freeness, rank formulas, Chern class computations, and the monoidal structure of $\Vmod$. \Cref{sec:twCBprop} summarizes the corresponding results in the twisted setting.

Finally, \cref{sec:BunG} explains the relationship between (twisted) conformal blocks and moduli of principal bundles. \Cref{sec:PBT} reviews parahoric Bruhat--Tits group schemes, while \cref{sec:PicBunG,sec:tCBforBunG} describe line bundles on the moduli stack of principal bundles and their global sections using twisted conformal blocks.

\subsection*{Acknowledgment} These notes stem from the lecture given at the Bootcamp for the 2025 Algebraic Geometry Summer Research Institute which took place at Colorado State University in July 2025. I am indebted to the Bootcamp organizers (Izzet Coşkun, Emanuele Macrì, Alexander Perry, Kevin Tucker and Isabel Vogt) for giving me the opportunity to be a mentor at this event.

I also want to thank the students and post-docs who attended my workshop: Avik Chakravarty (University of Pennsylvania), Nathaniel Collins (Colorado State University), Connor Halleck-Dube (UC Berkeley), Crislaine Kuster (IMPA), Davide Gori (Sapienza University of Rome), Hank Morris (University of Illinois at Chicago), Lillian McPherson (UC San Diego), Shuo Gao (Stony Brook University), Sixuan Lou (University of Illinois at Chicago).

Finally, I am grateful to my colleagues and collaborators. In particular, I thank David Ben-Zvi, Angela Gibney, Jiuzu Hong and Lukas Woike for their valuable mathematical insights,  and  the anonymous referee for their careful reading and many helpful suggestions.

\section{Conformal blocks}\label{sec:CB} A good way to understand the nature of conformal blocks is to recall their construction given by Tsuchiya, Ueno and Yamada \cite{tuy} which is formulated using representation theory of affine Lie algebras. In \cref{sec:CBfromLieAlgebras} we will present this special case and then, in \cref{sec:CBfromVOA}, we consider conformal blocks from representations of vertex operator algebras. Finally, in \cref{sec:twCB} we consider \textit{twisted} conformal blocks.

\subsection{Conformal blocks from Lie Algebras} \label{sec:CBfromLieAlgebras} We describe in this section conformal blocks arising from affine Lie algebras. We refer the reader to \cite{beauville:96:verlinde,sorger:96:verlinde,looijenga,kumar:book,belkale.fakhruddin}.
 Let $\g$ be a simple Lie algebra and fix a positive integer $\ell$. Then we consider the affine Lie algebra
\[ \gh \coloneqq \g \otimes \Cl{t} \oplus \CC \mathsf{K},
\] where $\mathsf{K}$ is a formal variable and with Lie bracket given by
\[ [Xf(t) + \alpha, Yg(t) + \beta] = [X,Y] f(t)g(t) + (X|Y) \Res(g(t)df(t)) \mathsf{K},
\] where the equality $2 \check{h}(\,|\,)=\kappa $ holds, for $\kappa$ the Killing form of $\g$ and $\check{h}$ the dual coxeter number of $\g$. The set of representations of $\gh$ we consider here, denoted $P^+_\ell$ is the set of highest weight irreducible representations of $\gh$ on which $\mathsf{K}$ acts by multiplication by $\ell$ and which are integrable (see \cite{kac:1990:infinite} for details). These representations are combinatorially encoded by affine weights $\Lambda$ of central charge $\ell$ (see \cref{ex:Lambda}). Following Kac \cite{kac:1990:infinite}, $P^+_\ell$ is in bijection with the set of finite dimensional irreducible representations of $\g$ of \textit{level} at most $\ell$ (that is those in which if $X \in \g$ is nilpotent, then $X^{\ell +1}$ acts trivially). Combinatorially, this is the set of dominant weights of $\g$ that, when paired with the coroot associated with the highest root of $\g$, yield an integer smaller than or equal to $\ell$. We observe that $P^+_\ell$ is closed under taking duals and throughout we will use the following notation
\[\xymatrix@R=1mm@C=5mm{
    & \text{weight} &  \text{representation of $\g$} &  \text{representation of $\gh$ of level $\ell$} \\
     \ar@(l,l)[d]^{\text{\small \qquad  dual}} & \lambda  & V_\lambda & L_{\ell,\lambda}(\g)=\Hc_{\lambda,\ell}=\bigoplus_{i \in \NN} \Hc_{\lambda,\ell}(i)  \\
     &\lambda^* & V_{\lambda^*}=V_\lambda^\vee & L_{\ell,\lambda}(\g)'=\Hc_{\lambda,\ell}'=\bigoplus_{i \in \NN} (\Hc_{\lambda,\ell}(i)^\vee)
}\]

\pp Let $C$ be a projective curve over $\CC$ which has at worst nodal singularities. Let $P_1, \dots, P_n$ be a collection of distinct and smooth points of $C$. We further assume that every irreducible component of $C$ is marked by at least one point $P_i$ or, equivalently, that $C \setminus P_\bullet$ is affine. For every $i$, by choosing local coordinates $t_i$ at $P_i$, we can interpret $\gh_i \coloneqq \g \otimes \Cl{t_i} \oplus \CC$ as a Lie algebra over the formal disk $\DD^\times_i$ around the point $P_i$.  Consider the Lie algebra $\g_{C\setminus P_\bullet} \coloneqq \g \otimes \Hm^0(C\setminus P_\bullet,\Oc_C)$ with Lie bracket $[X f, Yg] = [X,Y]fg$. In view of the residue theorem, the map \[\g_{C\setminus P_\bullet} \to \bigoplus_{i=1}^n \gh_i, \qquad X \otimes f \mapsto (X \otimes f_i)_{i=1}^n,
\] is a Lie algebra homomorphism. Therefore, the Lie algebra $\g_{C\setminus P_\bullet}$ acts on every representation of $\bigoplus_{i=1}^n \gh_i$.

\begin{defi}\cite{tuy} \label{def:CBtuy} The \emph{space of coinvariants} associated with $(C,P_1,\dots, P_n)$ and decorated by $\lambda_1, \dots, \lambda_n \in P^+_\ell$ is the quotient
\begin{equation*} \VV(C;P_\bullet;(\lambda_\bullet; \ell)) \coloneqq \dfrac{\Hc_{\lambda_1, \ell} \otimes \dots \otimes \Hc_{\lambda_n, \ell}}{\g_{C\setminus P_\bullet} \left(\Hc_{\lambda_1, \ell} \otimes \dots \otimes \Hc_{\lambda_n, \ell}\right)} = \dfrac{\Hc_{\lambda_\bullet,\ell}}{\g_{C\setminus P_\bullet} \Hc_{\lambda_\bullet,\ell}}.
\end{equation*} 
Similarly, we define the space of \emph{conformal blocks} $\VV(C;P_\bullet;\lambda_\bullet)^\dagger$ as the linear dual of the space of coinvariants, that is
\begin{equation*} \VV(C;P_\bullet;(\lambda_\bullet;\ell))^\dagger \coloneqq \VV(C;P_\bullet;(\lambda_\bullet;\ell))^\vee =\Hom_{\g_{C \setminus P_\bullet}}(\Hc_{\lambda_1,\ell} \otimes \dots \otimes \Hc_{\lambda_n,\ell}, \CC),
\end{equation*}  where $\CC$ is the trivial representation of $\g_{C \setminus P_\bullet}$.\end{defi} 

\pp We observe that, since the action of $\g_{C\setminus P_\bullet}$ on $\Hc_{\lambda_\bullet,\ell}$ depends on a choice of coordinates $t_i$ at the points $P_i$, then also the notion of conformal blocks depends on this datum. However Tsuchimoto shows that spaces of coinvariants and conformal blocks are independent of the choice of coordinates \cite{tsuchimoto} (see also \cite{bzf,DGT1,DGT2}).

\begin{ex} \label{ex:CBonP1} When $C=\PP^1$, one can show that 
$\VV(\PP^1; P_\bullet; (\lambda_\bullet;\ell))$ is naturally isomorphic to a quotient of $V_{\lambda_1} \otimes \dots \otimes V_{\lambda_n}$ by the action of $\g$ and by an operator depending on $\ell$ \cite{beauville:96:verlinde}. This provides a simple proof of finite dimensionality of spaces of coinvariants over $\PP^1$. As a consequence, one obtains
\[ \VV(\PP^1; P; (\lambda;\ell))= \left\lbrace \begin{array}{ll}
     \CC& \text{ if } \lambda =0 \\ 
     0 &  \text{ otherwise}
\end{array} \right.\] \and \[
\VV(\PP^1; P_1, P_2; (\lambda, \mu; \ell))= \left\lbrace \begin{array}{ll}
     \CC& \text{ if } \lambda = \mu^*  \\
     0 &  \text{ otherwise}
\end{array} \right. .
\]
\end{ex}

\subsection{Conformal blocks from VOAs} \label{sec:CBfromVOA} We describe in this section conformal blocks from vertex operator algebras (VOAs). We refer the interested reader to  \cite{lepli,bzf,DGT1,DGT2} and references therein for the precise definitions of VOAs and of their modules, while we give here a more intuitive description. 

\pp \label{pp:VOA}  Loosely speaking, a {VOA} $V$ is a simultaneous generalization of both a differential commutative algebra and of a Lie algebra. This is achieved by associating to every element $\bm{a} \in V$ a family $\left(\bm{a}_{(n)}\right)_{n\in\ZZ}$ of endomorphisms of $V$ which are to satisfy appropriate conditions, resembling associativity and commutativity.  More precisely, a VOA is a four-tuple $V=(V,Y, \vac, \conf)$ where $V= \bigoplus_{i\in \NN} V(i)$ is a graded  vector space and \[Y=Y( - , z) \colon V \to \End(V)[\![ z, z^{-1}]\!], \qquad \bm{a} \mapsto Y(a,z)\coloneqq \bm{a}(z)=\sum \bm{a}_{(n)}z^{-n-1}\] encodes the data of the endomorphisms $(\bm{a}_{(n)})_{n\in\ZZ}$. The elements $\vac \in V(0)$ and $\conf \in V_2$ are called the \textit{vacuum vector}---which plays a role analogous to that of a lowest weight vector or of an identity element---and the \textit{conformal vector}---which induces an action of the Virasoro Lie algebra $\Vir$ on $V$---respectively. The central charge\footnote{The Virasoro Lie algebra $\Vir$ is spanned by elements $L_p$, for $p \in \ZZ$ and by a central element $\mathsf{C}$ with Lie bracket \[ [L_p, L_q] = (p-q) L_{p+q} + \mathsf{C} \dfrac{p^3-p}{12}  \delta_{p+q,0}.\] An action of $\Vir$ has central charge equal to a constant $\bm{c}$ if $\mathsf{C}$ acts by multiplication by $\bm{c}$. Via the projection $L_p \mapsto -t^{p+1}\partial_t$, it is a central extension of the Witt Lie algebra $\Cl{t}\partial_t$.} of this action is denoted $c_V$ and called the central charge of $V$.  Throughout this note we will always assume that VOAs are of CFT-type and $C_1$-cofinite. These conditions imply in particular that $V(0)$ is spanned by $\vac$ and that the homogeneous vector spaces $V(i)$ are finite dimensional.

\pp \label{sec:VMod} Now that we have given a sense of what VOAs are, the set of representations associated with $V$ consists of \textit{admissible $V$-modules}, which we will call $V$-modules for simplicity. {Modules} over $V$ are, roughly speaking, $\NN$-graded vector spaces $M=\bigoplus_{i \in \NN} M(i)$ which are equipped with a map \[Y^M(\,,z) \colon V \to \End(M)[\![z,z^{-1}]\!], \qquad \bm{a} \mapsto \bm{a}^M(z) :=\sum_{n\in\ZZ} \bm{a}^M_{(n)}z^{-n-1}\] which ought to satisfy properties similar to those of $Y$. In particular $M$ has an induced action of the Lie algebra 
\[\Lk(V)=\dfrac{V \otimes \Cl{t}}{\nabla} \text{\qquad  where \qquad }\nabla =\omega_{(0)} \otimes \mathrm{Id}_{\CC(\!(t)\!)} + \mathrm{Id}_V \otimes \frac{d}{dt}.\]
In this way, it will be harmless for the reader who is not familiar with VOAs to think of a $V$-module as  representations for an infinite dimensional Lie algebra which is however not of the type $\g \otimes \Cl{t}$ in general.  
The admissibility condition that we assume here imposes that the homogeneous vector spaces $M(i)$ are finite dimensional. In particular one notes that $V$ is itself a $V$-module and that to every $V$-module $M$ one can associate a natural structure of $V$-module on $M' \coloneqq \bigoplus_{i \in \NN} (M{(i)})^\vee$ which is called the  \textit{contragredient} module of $M$.

\pp When the category of $V$-modules is semisimple and has finitely many simple objects, then we say that $V$ is \textit{rational}. A ``weaker'' condition is encoded in a property called $C_2$-cofiniteness\footnote{It is conjectured, but not known if rationality implies $C_2$-cofiniteness, so this is why ``weaker'' is in quotation marks.}, which requires that the space generated by $\bm{a}_{(-2)}(\bm{b})$ has finite codimension in $V$. This condition guarantees that $V$ has finitely many simple objects. When a VOA $V$ is $C_2$-cofinite but not rational, then the category of $V$-module has finitely many simple objects, but it is not semisimple (i.e. there exist indecomposable modules which are not simple). We call such VOAs \textit{logarithmic}. If $V$ is $C_2$-cofinite, rational and it is self-contragredient, i.e. $V=V'$, then $V$ is called \textit{strongly rational}.

\begin{ex} \label{ex:Vlg} Starting with a simple Lie algebra $\g$, one can show that  the representation $L_{\ell}(\g)\coloneqq L_{\ell,0}(\g)$ from \ref{sec:CBfromLieAlgebras} is actually a rational and $C_2$-cofinite VOA and that $L_{\ell,\lambda}(\g)$ are all its simple modules. In the VOA literature $L_\ell(\g)$ is called the \textit{simple affine VOA of level $\ell$}. 

In fact $L_{\ell}(\g)$ is the simple quotient of the universal Verma module $V_\ell(\g)$ for $\gh$, which is itself a VOA (and which can be constructed for every complex number $\ell \neq - \check{h}$). Although it is neither rational nor $C_2$-cofinite, $V_\ell(\g)$ is easier to describe than $L_\ell(\g)$. In fact, one has \[V_\ell(\g) = U(\gh) / (\g \otimes t\CC[\![t]\!], \mathsf{K} = \ell) \cong U(\g \otimes t^{-1}\CC[t^{-1}]),\] where $U(-)$ denotes the universal enveloping algebra and where the last isomorphism is an isomorphism of vector spaces. It follows that every element of $V_\ell(\g)$ is a linear combination of elements of the form $X_1 t^{r_1} \circ \dots \circ X_s t^{r_s}$, for $X_i \in \g$ and $r_i \leq -1$, and for $s \in \NN$. The element $1$ is the vacuum element of $V_\ell(\g)$ and one can show that the conformal vector $\conf$ is given by the element 
\[ \dfrac{1}{2(\check{h}+\ell)}\sum_{\alpha} X_\alpha t^{-1} \circ X_\alpha t^{-1},
\] where $\{X_\alpha\}$ is an orthonormal basis of $\g$ with respect to the form $(\,|\,)$. One also has
\[ (X t^{-1})(z) = \sum_{n \in \ZZ} (Xt^n\circ -) z^{-n-1},
\] where the endomorphism $(Xt^n\circ -)$ is left multiplication by the element $Xt^n$. The fields associated with the other elements of $V_\ell(\g)$ can be similarly described.  

A similar construction to $V_\ell(\g)$, replacing $\g$ with an abelian Lie algebra $\h$ of dimension $r$ yields the Heisenberg VOA of rank $r$.
\end{ex}

\pp Consider now a pointed curve $(C, P_1, \dots, P_n)$ decorated with $V$-modules $M_1, \dots, M_n$. We explain briefly how to associate conformal blocks to this data. As in the previous section, we first require that each irreducible component of $C$ is marked by at least one point $P_i$ and we also choose local coordinates $t_i$ at every $P_i$. 

For the VOA $L_\ell(\g)$, one could easily define the Lie algebra $\g_{C \setminus P_\bullet}$ which encoded both the geometry of $(C;P_\bullet)$ and the structure of $\gh$-modules of each decoration. For a general VOA $V$, however, the construction of such a Lie algebra is much more subtle and requires constructing a \textit{global version} of the Lie algebra $\Lk(V)$, which we can interpret as playing an analogous role to that of $\gh$. We postpone the construction of this Lie algebra, denoted $\Lc_{C\setminus P_\bullet}(V)$, to \cref{rmk:LCPVconstruction} \ref{it:LCPV1} and following \cite{DGT2} we call it the \textit{chiral Lie algebra} of $V$ over $C$.

\begin{defi}\label{def:CB}
    The \emph{space of coinvariants} associated with $(C;P_\bullet; M_\bullet)$ is
    \[ \VV(C;P_\bullet; M_\bullet) \coloneqq \dfrac{M_1 \otimes \cdots \otimes M_n}{\Lc_{C\setminus P_\bullet}(V) \left(M_1 \otimes \cdots \otimes M_n\right)} = \dfrac{M_\bullet}{\Lc_{C\setminus P_\bullet}(V)  M_\bullet}.
    \] The space of \emph{conformal blocks} $\VV(C;P_\bullet; M_\bullet)^\dagger$ is the linear dual of $\VV(C;P_\bullet; M_\bullet)$.
\end{defi}

\begin{rmk} \label{rmk:LCPVconstruction} \begin{enumerate}[label=(\alph*)] \item \label{it:LCPV1} The construction of $\Lc_{C\setminus P_\bullet}(V)$ is achieved by first constructing a global version of $V$ over $C$, which we denote $\Vc_C$. The action of the Virasoro Lie algebra on $V$ induces a connection $\nabla \colon \Vc_C \to \Vc_C \otimes \omega_C$ and so one may define
\begin{equation} \label{eq:LCPV} \Lc_{C\setminus P_\bullet}(V) \coloneqq \Ho^0\left(C\setminus P_\bullet, \dfrac{\Vc_C \otimes \omega_C}{\operatorname{im} \nabla}\right).\end{equation}
This definition has been given in \cite{bzf} for smooth curves and extended to also nodal curves in \cite{DGT1,DGT2}. For rational curves \cite{NT} consider a different Lie algebra which however gives rise to the same spaces of coinvariants \cite{DGT2}. 
\item \label{it:LCPV2}  When $V=L_\ell(\g)$ the Lie algebra $\Lc_{C\setminus P_\bullet}(V)$ and $\g_{C\setminus P_\bullet}$ are not isomorphic but, as shown in \cite{bzf}, they define the same spaces of coinvariants. This is due to the fact that $L_\ell(\g)$ is generated by its degree 1 which is naturally in bijection with $\g$. 
\end{enumerate}
\end{rmk}

\subsection{Twisted conformal blocks} \label{sec:twCB} Starting with a VOA $V$, one might wish to extend the notion of conformal blocks to a wider class of $V$-modules. The case that we discuss here, and which will be relevant in \cref{sec:tCBforBunG}, concerns the case of \emph{twisted} $V$-modules \cite{Dong:twisted}.

\pp Assume that $V$ is equipped with an automorphism $\sigma$ of finite order, let's say $m$. To this datum, one can define the concept of $\sigma$-twisted $V$-modules.  A $\sigma$-twisted $V$-module is a vector space $M$ equipped with a map
\[ Y^{M}(\, , z) \colon V \to \End(M)[\![ z^{1/m}, z^{-1/m}]\!], \qquad \bm{a} \mapsto \bm{a}^{M}(z)
\] satisfying appropriate conditions---similar to those satisfied by untwisted $V$-modules---and in particular such that, whenever $\sigma(\bm{v}) = e^{\frac{2\pi i s}{m}}\bm{v}$, then $\bm{v}^{M}_{(n)}=0$ unless $n \in s/m + \ZZ$. It follows that, if a $V$-module has an action of the Lie algebra $\Lk(V)$, then a $\sigma$-twisted $V$-module has an action of the Lie algebra
\[\Lk^\sigma\!(V) = \dfrac{(V \otimes \Cl{t^{1/m}})^{\sigma}}{\nabla}.\]

\begin{eg}
    Let $V=L_{\ell}(\g)$ and let $\sigma$ be an automorphism of $\g$ of order $m$. Let $\zeta_m$ be a primitive $m$-th root of unity, and extend $\sigma$ to $\gh$ by sending $t$ to $\zeta_m^{-1}t$. This automorphism induces an automorphism of the VOA $L_\ell(\g)$ which we still denote $\sigma$. In the same spirit of \cref{ex:Vlg}, we can interpret $\sigma$-twisted $L_\ell(\g)$-modules as modules for the Lie algebra $\gh^\sigma$ obtained taking the $\sigma$-invariant elements of $\gh$. 
    As in the untwisted setting, also in this case the simple modules for $L_\ell(\g)$ can be given a combinatorial description in terms of affine weights $\Lambda$s of the \textit{twisted} affine Lie algebra $\gh^\sigma$ (see \cref{ex:Lambda}). Moreover, following \cite{kac:1990:infinite}, one can actually reduce the study of $\sigma$-twisted $L_\ell(\g)$-modules to the case in which $\sigma$ acts on $\g$ by diagram automorphisms only, so that only finitely many cases can be analyzed. 
\end{eg}

\pp The geometric counterpart of twisted modules---necessary to define twisted conformal blocks---does not involve pointed curves, but rather covers of such curves. Consider a $\Gamma$-cover of smooth and projective curves $\pi \colon X \to C$, where $\Gamma$ is a finite group.
Assume furthermore that $\Gamma$ acts on $V$ as a group of VOA automorphisms. 
Let $P_\bullet$ be a non-empty and finite set of points of $C$ and, for every $i$, choose $x_i \in \pi^{-1}P_i$ and denote $\Gamma_i$ the stabilizer of $x_i$, which is necessarily a cyclic group and acts on $V$. Let $M_i$ be a $\Gamma_i$-twisted $V$-module. In order to define conformal blocks associated with this datum, one needs to define a Lie algebra $\Lc_{\pi \colon X \to C, P_\bullet}(V)$ analogous to $\Lc_{C \setminus P_\bullet}(V)$ which globalizes $\Lk^{\Gamma_i}(V)$ and therefore acts on every $M_i$. With the notation of \cref{rmk:LCPVconstruction}, one first considers a twisted version $\Vc$ over $C$ by defining $\widetilde{\Vc}^\pi_C=(\pi_*\Vc_X)^\Gamma$. Namely, a section of $\widetilde{\Vc}^\pi_C$ is a $\Gamma$-invariant section of $\Vc_X$. Then $\Lc_{\pi \colon X \to C, P_\bullet}(V)$ is defined as in \eqref{eq:LCPV} by replacing $\Vc_C$ with $\widetilde{\Vc}^\pi_C$. 

\begin{defi} The space of \emph{twisted coinvariants} is defined  as the quotient
\[ \VV\left(\pi \colon X \to C; x_\bullet;M_\bullet\right) = \dfrac{M_1 \otimes \dots \otimes M_n}{\Lc_{\pi \colon X \to C, P_\bullet}(V) (M_1 \otimes \dots \otimes M_n)}= \dfrac{M_\bullet}{\Lc_{\pi \colon X \to C, P_\bullet}(V) M_\bullet.}
\]  The space of \emph{twisted conformal blocks} is the linear dual of the space of coinvariants and it is denoted $\VV\left(\pi \colon X \to C; x_\bullet;M_\bullet\right)^\dagger$.\end{defi}

The definition given above is due to \cite{frenkelSzczesny:twusted}, where the authors also discuss independence on the choice of coordinate at $x_\bullet$.

\begin{ex} As discussed in \cref{rmk:LCPVconstruction} \ref{it:LCPV2}, when $V=L_\ell(\g)$, spaces of coinvariants can be taken with respect to $\Lc_{C\setminus P_\bullet}(V)$ or $\g_{C \setminus P}$, yielding the same vector space. Similarly for twisted conformal blocks, when $V=L_\ell(\g)$, one can replace the global Lie algebra $\Lc_{\pi \colon X \to C, P_\bullet}(V)$ with $\g_{C,P_\bullet}^\Gamma$ the $\Gamma$-invariant sections of the Lie algebra $\g \otimes \Ho^0(X\setminus \pi^{-1}(P_\bullet),\Oc_X)$ \cite{frenkelSzczesny:twusted}.   
\end{ex}

\begin{rmk}\label{rmk:twCBrmk} We saw that untwisted coinvariants can be defined also for nodal curves. The definition of twisted coinvariants for covers of nodal curves is however not known in general. When $V = L_\ell(\g)$, associated twisted conformal blocks have been extended also to admissible stable covers \cite{damiolini:2020:conformal,hongkumar:2023, deshpande.mukhopadhyay:2019,hongkumar:2024} using the language of affine Lie algebras. 
\end{rmk}

\section{Properties of conformal blocks}\label{sec:properties} We describe in this section the properties that (twisted) conformal blocks from VOAs satisfy, whose consequences are discussed in \cref{sec:cor}. In particular, one obtains that when $V$ is strongly rational, then spaces of coinvariants fit together to define a vector bundle over $\bMgn$ whose rank and Chern classes can be explicitly computed (see \cref{sec:VB,sec:chern}). When $V=L_\ell(\g)$ we refer the reader to \cite{beauville:96:verlinde, sorger:96:verlinde,looijenga, Fakhruddin,kumar:book,belkale.fakhruddin}.

\subsection{Sheaves on moduli of curves} \label{sec:sheaf} In order to correctly define spaces of coinvariants, we imposed that every irreducible component of $C$ had to be marked by a point (or equivalently that $C \setminus P_\bullet$ was an affine curve). In order to construct a sheaf over the moduli space $\bMgn$ of stable pointed curves, we need to be able to remove this requirement. In fact, one may remove this condition using the following theorem, which goes under the name of \emph{propagation of vacua}.

\begin{thm} \label{thm:POV} The map 
\[M_1 \otimes \dots \otimes M_n \to M_1 \otimes \dots \otimes M_n \otimes V, \qquad \bm{m}_\bullet \mapsto \bm{m}_\bullet \otimes \vac\] induces an isomorphism
\begin{equation} \label{eq:POV} \VV(C;P_\bullet; M_\bullet) \cong \VV(C; P_\bullet, Q; M_\bullet, V).
\end{equation} for every smooth point $Q \in C \setminus P_\bullet$. \end{thm}

This result has been proved for $V= L_\ell(\g)$ by \cite{tuy} and generalized to the case of VOAs in \cite{bzf, NT,  DGT1}. In the analytic setting, the result follows from \cite{gui:sewingandprop}. Having this in hand allows us to define $\VV(C;P_\bullet; M_\bullet)$ for every coordinatized pointed curve $(C;P_\bullet; t_\bullet)$.

\pp Spaces of coinvariants can be defined not only for a single coordinatized pointed curve $(C;P_\bullet;t_\bullet)$, but also for families of such curves. This yields a quasi-coherent sheaf  $\VV_g(M_\bullet)$ on the stack $\btMgn$ parametrizing coordinatized pointed curves of genus $g$ which we call the \emph{sheaf of coinvariants}. The dual sheaf, not necessarily quasi coherent, is the sheaf of conformal blocks $\VV_g(M_\bullet)^\dagger$.  Actually, one can show that these sheaves descend to a quasi-coherent sheaf on $\bTMgn$, the stack parametrizing  pointed curves $(C;P_\bullet)$ together with a non-zero tangent vector $\tau_i$ at each point $P_i$ and such that the group of automorphisms of the triple $(C;P_\bullet; \tau_\bullet)$ is finite. Unfortunately, for a general VOA $V$, we do not know to which extent these sheaves descend to the moduli stack of stable curves $\bMgn$ (without the extra data at the marked points).  

The principal factor that controls the descent of these sheaves from $\btMgn$ to $\bTMgn$ and then to $\bMgn$ is the \textit{integrability} of the action of the non-negative part of the Virasoro algebra on the $V$-modules $M_\bullet$. When $p \geq 1$, the endomorphisms $L_p$ act on $M_\bullet$ as operator of negative degree and can therefore be integrated. This yields the first descent from $\btMgn$ to $\bTMgn$. The descent from $\bTMgn$ to $\bMgn$ is controlled by $L_0$, whose integrability is more delicate. If the $V$-modules $M_i$ are such that $L_0$ acts on its degree $d$-component $M_i(d)$ as multiplication by a scalar $d+{a}_i$ for a \textit{rational} number ${a}_i \in \QQ$, then one can prove that the sheaf $\VV_g(M_\bullet)$ descends to $\bMgn$ (see \cite[Section 8]{DGT2}).  

\begin{prop} If $V$ is rational (e.g. $V=L_\ell(\g)$), then $\VV_g(M_\bullet)$ descends to $\bMgn$. If $V$ is $C_2$-cofinite and $M_1,\dots, M_n$ are simple $V$-modules, then $\VV_g(M_\bullet)$ descends to $\bMgn$. \end{prop}

With the exception of \cref{prop:connectionTMgn} and \cref{sec:MF}, the reader can safely read all the statements pretending that sheaves of coinvariants always descend to $\bMgn$.

\subsection{Coherence} Spaces of coinvariants are not finite dimensional in general. For instance, one can show that spaces of conformal blocks associated with the Heisenberg VOA are finite dimensional over rational curves, but infinite dimensional for higher genus curves \cite{bzf}. Since coherent sheaves are more tractable than quasi-coherent ones, it is convenient to identify cases under which spaces of conformal blocks are finite dimensional.  We summarize the current state of the art in the following statement:

\begin{prop} \label{prop:coherence} \begin{enumerate}[label=(\alph*)]
    \item \label{it:coh1} If $V$ is $C_2$-cofinite, then $\VV_g(M_\bullet)$ is coherent on $\bMgn$. 
    \item \label{it:coh2} If $V$ is generated in degree 1 then $\VV_0(M_\bullet)$  is a quotient of the constant vector bundle with fiber $M_1(0) \otimes \dots \otimes M_n(0)$. Therefore $\VV_0(M_\bullet)$ is a globally generated coherent sheaf.
    \end{enumerate}
\end{prop}

\Cref{prop:coherence} \ref{it:coh1} was proved in \cite{DGK1}, generalizing the analogous result from \cite{DGT2}, where rationality was also assumed. The result in the case of $V=L_\ell(\g)$ was already proved in \cite{tuy}, while the coherence on $\Mgn$ was established by \cite{abe.nagatomo}.

\Cref{prop:coherence} \ref{it:coh2}, proved in \cite{DG}, can be interpreted as a generalization of the analogous statement for the VOA $V=L_\ell(\g)$---indeed generated in degree 1. For $L_\ell(\g)$, this is a consequence of the description of spaces of coinvariants given in \cref{ex:CBonP1}, and extended to families of rational pointed curves in \cite{Fakhruddin}. In fact, for every $L_\ell(\g)$-module $M=L_{\ell,\lambda}$ one has that $M(0)=V_\lambda$. By computing the Chern classes of sheaves of coinvariants, in \cite{DG} the author and Gibney provide examples of sheaves of coinvariants which are not globally generated in genus zero (see \cref{sec:chern}).

\subsection{Factorization and Sewing} \label{sec:sewing} We now consider the situation in which $(C,P_\bullet)$ is a stable pointed curve with (at least) one node $Q$. Denote by $(\Ct;P_\bullet, Q_\pm)$ the partial normalization of $C$ obtained resolving the singularity at $Q$, and where $Q_\pm$ denote the two points lying above the node $Q$. Recall that if $W$ is a $V$-module, we denote by $W'$ its contragredient. The \emph{factorization theorem} asserts the following.

\begin{thm} \label{thm:factorization}  Let $V$ be a rational VOA and denote by $\Ck_V$ the set of simple $V$-modules. Then the map
\[M_1 \otimes \dots \otimes M_n \to M_1 \otimes \dots \otimes M_n \otimes W \otimes W', \qquad \bm{m}_\bullet \mapsto \bm{m}_\bullet \otimes \id_{W(0)}\] 
induces an isomorphism
\begin{equation} \label{eq:fact}  \VV(C;P_\bullet;M_\bullet) \cong  \bigoplus_{W \in \Ck_V}  \VV\left(\Ct; P_\bullet, Q_\pm; M_\bullet, W, W'\right).
\end{equation}
\end{thm}
When $V=L_\ell(\g)$, this has been proved in \cite{tuy} and, in the unpublished but foundational manuscript \cite{BFM}, it is considered the case of the Virasoro VOA. For rational VOAs, the theorem has been proved for rational curves in \cite{NT}, and for every genus in \cite{DGT2}.

\medskip 

We note that since $V$ is rational, the set $\Ck_V$ indexing the direct sum on the right hand side is finite. Furthermore, on the right hand side, we can interpret the point $Q_+$ being decorated by $W$ and the point $Q_-$ by $W'$. However, we can also interpret the divisor $Q_+ + Q_-$ to be decorated by the $V$-bimodule $\bigoplus_{W \in \Ck_V}W \otimes W'$ and so we can write 
\[ \bigoplus_{W \in \Ck_V}  \VV(\Ct; P_\bullet, Q_\pm; M_\bullet, W, W') \cong \VV\left(\Ct; P_\bullet, Q_\pm; M_\bullet, \bigoplus_{W \in \Ck_V}W \otimes W'\right).
\] 

\pp \label{pp:MTA} When $V$ is not rational, an appropriate modification of this $V$-bimodule is necessary to obtain an analogue of \cref{thm:factorization}. In fact, in \cite{DGK1} it is shown that, for every VOA $V$ which is $C_1$-cofinite and of CFT type (as those considered here), one has an isomorphism
\begin{equation} \label{eq:factMTA} \VV(C;P_\bullet;M_\bullet) \cong \VV\left(\Ct; P_\bullet, Q_\pm; M_\bullet, \Ak(V)\right),
\end{equation} where $\Ak(V)$ is the \emph{mode transition algebra} of $V$. The algebra $\Ak(V)=\bigoplus_{(d,e) \in \NN^2}\Ak_{d,e}$---which is a bi-graded $V$-bimodule---has been formally introduced in \cite{DGK2} and, similarly to the Zhu algebras $\Aa_i(V)$, to which it is closely related, it packages much of the ``\textit{highest weight}'' representation theory of $V$ (see also \cite{DGK3}). The isomorphism \eqref{eq:factMTA} is induced by 
\[ M \to M \otimes \Ak(V), \qquad \bm{m}_\bullet \mapsto \bm{m}_\bullet \otimes \Ic_0,
\]where $\Ic_0 \in \Ak_{0,0}$ is a specified element arising from the vacuum element of $V$. When $V$ is rational, then $\Ak$ decomposes as   $\bigoplus_{W \in \Ck_V}W \otimes W'$ and $\Ic_0 = \sum \id_{W(0)}$.

\pp We describe here a more refined version of \cref{thm:factorization}, called \emph{sewing} or \emph{smoothing} theorem which will be crucial to extend local properties of $\VV_g(M_\bullet)$ from $\TMgn$ (or $\Mgn$) to the whole $\bTMgn$ (resp. $\bMgn$). To explain this, we will set up some notation and, for simplicity of notation, use the moduli space $\bMgn$ rather than $\bTMgn$.

Recall that $\bMgn$ admits a stratification by opens
\[ \Mgn=\bMgn^{(0)} \subset\bMgn^{(1)} \subset \dots \subset \bMgn^{(3g+n-4)}\subset \bMgn^{(3g+n-3)}= \bMgn,
\] where $\bMgn^{(k)}$ parametrizes those curves which have at most $k$ nodes. 

Let $x \in \bMgn^{(k)}\setminus \bMgn^{(k-1)}$ be a point which corresponds to a curve $C$ which has exactly $k$ nodes. One can associate to such a curve and to one of the nodes---denote it $Q$---two families of curves over $\DD \coloneqq \Spec(\CC[\![t]\!])=\{ 0, \eta\}$ whose spaces of coinvariants will be related:
\begin{enumerate}[label=(\Roman*)]
    \item a \textit{smoothing} family $(\Cc,\Pc_\bullet)$, such that $\Cc|_\eta$ has $k-1$ nodes, $\Cc|_0=C$ and, locally around $Q$, $\Cc$ has equation $xy=t$; This implies that the curve $\Cc \to \Delta$ corresponds to a map $\Delta \to \bMgn^{(k)}$ whose closed point, $C$, maps to $\bMgn^{(k)}$, but which generically lands in $\bMgn^{(k-1)}$. 
    \item a \textit{trivial} family $(\Ct[\![t]\!]; P_\bullet, Q_\pm)$, where $\Ct$ is the partial normalization of $C$ obtained resolving the singularity at $Q$, and where $Q_\pm$ denote the two points lying above the node $Q$. Assuming that $Q$ is a non-separating node, then this family corresponds to a map $\DD \to \bM_{g-1,n+2}^{(k-1)}$. If $Q$ is a separating node, then it would correspond to a map $\DD \to \bM_{g_1,n_1+1}^{(k-1)} \times \bM_{g_2, n_2+1}$ for appropriate $g_i$ and $n_i$ such that $g_1+g_2=g$ and $n_1+n_2=n$.
\end{enumerate}
Now that we have set up the notation for the geometric input of conformal blocks, we are ready to state the \emph{sewing theorem}.

\begin{thm}\label{thm:sewing} \cite{tuy,BFM,NT,DGT2}
    Let $V$ be a rational and $C_2$-cofinite VOA. Then the map 
    \begin{equation} \label{eq:mapsew} M_\bullet \to M_\bullet \otimes (W \otimes W')[\![t]\!], \qquad \bm{m}_\bullet \mapsto \bm{m}_\bullet \otimes \sum_{d \in \NN} \id_{W(d)} \, t^d
    \end{equation} induces an isomorphism 
    \begin{equation} \label{eq:sew} \VV (\Cc, \Pc_\bullet;M_\bullet) \cong \bigoplus_{W \in \Ck_V} \VV\left(\widetilde{C}[\![t]\!]; P_\bullet, Q_\pm;M_\bullet, W, W'\right).
\end{equation} 
\end{thm}

Once \cref{thm:factorization} is proved, in order to prove \cref{thm:sewing} it is enough to show that the map \eqref{eq:mapsew} indeed descends to a map between coinvariants. By Nakayama's Lemma, and using the fact that coinvariants are finitely generated (by the $C_2$-cofiniteness), if such a map exists, it necessarily must be an isomorphism since it specializes to an isomorphism---namely to \eqref{eq:fact}---when $t=0$.

 As we discussed earlier, \cref{thm:factorization} could be generalized beyond the rational case, by replacing $\bigoplus W \otimes W'$ with the mode transition algebra $\Ak(V)$. In order to obtain a general version of \cref{thm:sewing}, one would wish to find elements $\Ic_{d}$ such that the map
\[ M^\bullet \to M^\bullet \otimes \Ak[\![t]\!], \qquad m_\bullet \to m_\bullet \otimes \sum_{d \in \NN} \Ic_d t^d
\] induces a map between the appropriate spaces of coinvariants. In \cite{DGK2} it is shown that this is possible if and only if $\Ak(V)$ satisfies an explicit algebraic condition called \textit{strong unitality}. Explicitly, $\Ak(V)$ is strongly unital if, for every $d$, there exist elements $\Ic_d \in \Ak(V)_{d,d}$ such that $\Ic_d \star \mathfrak{a} = \mathfrak{a}$ and $\mathfrak{b} = \mathfrak{b} \star \Ic_d$ for every $\mathfrak{a} \in \Ak(V)_{d,e}$  and $\mathfrak{b} \in \Ak(V)_{e,d}$. If $V$ is rational, then $\Ak(V)$ is necessarily strongly unital: from the decomposition $\Ak =\bigoplus_{W \in \Ck_V}W \otimes W'$ one obtains $\Ak_{d,e} = \bigoplus_{W \in \Ck_V} \Hom(W(d), W(e))$ and the element $\Ic_d = \sum_{W \in \Ck_V} \text{Id}_{W(d)}$ satisfies the strong unitality condition described above.  

The above observations yield the following generalization of \cref{thm:sewing}.

\begin{thm} \label{thm:sewingMTA}
    Let $V$ be a VOA with strongly unital mode transition algebra $\Ak(V)$ and with finite dimensional spaces of coinvariants. Then there exists an element $\Ic(t) \in \Ac[\![t]\!]$ such that the map
    \[ M_\bullet \to M_\bullet \otimes \Ak(V)[\![t]\!], \qquad \bm{m}_\bullet \mapsto \bm{m}_\bullet \otimes \Ic(t)
    \] induces an isomorphism 
    \begin{equation} \label{eq:sewMTA} \VV (\Cc, \Pc_\bullet;M_\bullet) \cong  \VV\left(\widetilde{C}[\![t]\!]; P_\bullet, Q_\pm;M_\bullet, \Ak(V)\right).
\end{equation} 
\end{thm}

\subsection{Connection}\label{sec:projconn} Using the Virasoro uniformization theorem, the action of the Virasoro algebra on $V$-modules induces a \textit{projectively flat connection with logarithmic singularities along the boundary} $\TDelta_{g,n}=\bTMgn \setminus \TMgn$. Recall that a flat connection on a sheaf $\VV$  over a space $X$ is the same as an action of the tangent bundle $\Tc_X$---seen as a sheaf of Lie algebras---on $\VV$. A flat connection with logarithmic singularities along a divisor $\Delta \subseteq X$ on $\VV$ is an action of $\Tc_X(-\Delta)$ on $\VV$. That is, not every vector field acts on $\VV$, but only those which are tangent to $\Delta$.  Finally, by a projectively flat connection with logarithmic singularities along $\Delta$ on $\VV$ we mean an action on $\VV$ not of $\Tc_X(-\Delta)$, but of a central extension $\Ac$ of $\Tc_X(-\Delta)$. Such extensions naturally arise from line bundles. 

Given a line bundle $\Lc$ on $X$ and an integer $\alpha$ one may define the sheaf $\alpha \Ac_{\Lc}$ as the sheaf of first order differential operators on the bundle $\Lc^{\otimes \alpha}$. This notion can naturally be extended to every $\alpha \in \CC$ (see \cite{tsuchimoto}) and yields an exact sequence
\[ 0 \longrightarrow \Oc_{X} \longrightarrow  \alpha \Ac_{\Lc} \longrightarrow \Tc_X \longrightarrow 0.
\]By restricting this sequence to $\Tc_X(-\Delta)$ one obtains the sheaf of Lie algebras $\alpha \Ac_{\Lc}(-\Delta)$ and the exact sequence
\begin{equation}\label{eq:atiyahLog}  0 \longrightarrow \Oc_{X} \longrightarrow  \alpha \Ac_{\Lc} (-\Delta) \longrightarrow \Tc_X(-\Delta) \longrightarrow 0.
\end{equation} 

\begin{prop} \label{prop:connectionTMgn} The sheaf $\VV_g(M_\bullet)$ on $\bTMgn$ is equipped with an action of the Atiyah algebra
\[ \dfrac{c_V}{2}\Ac_{\Lambda}(-\TDelta_{g,n}),
\]  where $c_V$ is the central charge of $V$ and where $\Lambda$ is the Hodge line bundle on $\bTMgn$. \end{prop}

Assume now that every $V$-module $M_i$ is such that the Virasoro element $L_0$ acts on its degree $d$-component $M_i(d)$ as multiplication by the scalar $d+{a}_i$ for a rational number ${a}_i \in \QQ$ ($a_i$ is called the conformal weight of $M_i$). This is the case if $M$ is a simple module for a $C_2$-cofinite VOA. As mentioned in \cref{sec:sheaf}, under this assumption, then the sheaf $\VV_g(M_\bullet)$ descends to $\bMgn$. Under these hypotheses one obtains a version of \cref{prop:connectionTMgn} over $\bMgn$.

\begin{prop} \label{prop:connectionMgn} \label{pp:connMgn} Under the above assumptions and notation, the sheaf $\VV_g(M_\bullet)$ on $\bMgn$ is equipped with an action of the Atiyah algebra
\[ \left(\dfrac{c_V}{2}\Ac_{\Lambda}+ \sum_{i=1}^n {a}_i\Ac_{\Psi_i}\right)(-\Delta_{g,n}),
\]  where $\Psi_i$ is the cotangent line bundle at the $i$-th marked point. \end{prop}

\subsection{Corollaries} \label{sec:cor} We collect in this section properties of conformal blocks that can be deduced as consequences of the properties described in the previous sections.  In \cref{sec:twCBprop} we summarize the properties of twisted conformal blocks.

\subsubsection{Vector bundles from coinvariants}  \label{sec:VB} Since coherent sheaves with a projectively flat connection are locally free, it follows that, whenever sheaves of coinvariants are coherent (e.g.  \cref{prop:coherence}), then they are actually locally free on the locus of smooth curves $\TMgn$ (or $\Mgn$ if in the sheaf descends as in the assumptions of \cref{prop:connectionMgn}). Since the connection on sheaves of coinvariants has logarithmic singularities along $\TDelta_{g,n}$ (resp. on $\Delta_{g,n}$), one cannot automatically extend this result to $\bTMgn$ (resp. to $\bMgn$). The sewing theorems from \cref{sec:sewing}, however, allow us to obtain local freeness on the whole $\bTMgn$ (resp. $\bMgn$) from local freeness on $\TMgn$ (resp. $\Mgn$). In particular we obtain the following result, which we state for sheaves over $\bTMgn$:

\begin{cor}\label{cor:vb} If the assumptions of \cref{thm:sewingMTA} are verified (that is $\Ak(V)$ is strongly unital and spaces of coinvariants are finite dimensional), then  $\VV_g(M_\bullet)$ is a vector bundle of finite rank on $\bTMgn$. In particular, if $V$ is $C_2$-cofinite and rational, then $\VV_g(M_\bullet)$ is a vector bundle of finite rank.
\end{cor}

An open question is to determine when coinvariants are locally free beyond the assumptions of \cref{cor:vb}. In particular, in \cite{DGK2} it is shown that the mode transition algebra of certain logarithmic VOAs are not strongly unital, so that \cref{thm:sewingMTA} cannot be used. 

\pp Assume now that $V$ is \textit{strongly rational}, i.e. $V$ is a self-contragredient $C_2$-cofinite and rational VOA. Without loss of generality assume that the modules $M_\bullet$ are simple. Then one can iteratively use \cref{thm:factorization} to reduce the computation of $\dim \VV(C;P_\bullet;M_\bullet)$ to that of the \textit{fusion rules}, i.e. to $\dim \VV(\PP^1;0, 1, \infty; A, B,C)$ for every triple $(A, B, C)$ of simple $V$-modules.  Moreover \cite{zhu.global} shows that  
\[\VV(\PP^1;P;M) = \CC \delta_{M=V} \qquad \VV(\PP^1; P_1, P_2; M, W) = \Hom_V(M,W') =\CC \delta_{M=W'}.\] 
\newcommand{\rk}{\operatorname{rk}}
\begin{ex} When $V$ is strongly rational we can readily compute the rank of $\VV_1(V)$:
\begin{align*}
    \rk \VV_1(V) &=  \sum_{W \in \Ck_V}  \rk  \VV_0(V,W,W') = \sum_{W \in \Ck_V} \dim \VV(\PP^1; 1,0,\infty; V,W,W')\\
    &= \sum_{W \in \Ck_V} \dim \VV(\PP^1; 0,\infty; W,W') = \sum_{W \in \Ck_V} \dim \CC = |\Ck_V|.
\end{align*}
\end{ex}

\subsubsection{Chern Character} \label{sec:chern} When $V$ is strongly rational, an application of the reconstruction theorem for \textit{semisimple} Cohomological Field Theories \cite{teleman2012structure} shows that it is possible to recover the Chern character of $\VV_g(M_\bullet)$ over $\bMgn$ from its first Chern class and its rank \cite{MOPPZ, DGT3}.  As in \cref{pp:connMgn}, denote by $a_i$ the conformal weight of $M_i$. Recall that the boundary divisor $\Delta_{g,n}$ is a union $\Delta_{\rm irr} \cup \bigcup_{i,I}\Delta_{i:I}$, where the union is over $i,I$ such that $i\in\{0,\dots,g\}$ and $I\subseteq \{1, \dots, n\}$, modulo the relation $(i,I)\equiv(g-i,I^c)$. Following standard conventions, we set $\lambda = c_1 \Lambda$, $\psi_i = c_1 \Psi_i$ and $\delta_* = c_1 \Delta_*$.
Then one has the expression
\begin{equation} \label{eq:c1}
c_1\left(  \VV_g(M_1, \dots, M_n)  \right) =  {\rk} \VV_g(M_\bullet)\left( \frac{c_V}{2}\lambda + \sum_{i=1}^n a_i \psi_i \right)          
-b_{\rm irr} \delta_{\rm irr} - \sum_{i,I} b_{i:I} \delta_{i:I}
\end{equation}
where                 
\begin{align*} 
b_{\rm irr} & =\sum_{W\in \Ck_V} a_W \cdot{\rk} \VV_{g-1}\left(M
_{\bullet}, W ,W'\right) \\               
b_{i:I} &=\sum_{W\in \Ck_V} a_W \cdot {\rk}\VV_{i}\left(M_I, W \right)\cdot {\rk}\VV_{g-i} \left(M_{I^c}, W' \right). 
\end{align*}
When $V=L_\ell(\g)$, the expression given in \eqref{eq:c1} can be found in \cite{Fakhruddin, mukhopadhyay:ranklevel}. For a strongly rational VOA, \eqref{eq:c1} is derived in \cite{DGT3}.

\medskip 
Since the fusion rules are determined for certain classes of VOAs, it is therefore possible to explicitly compute $\DD_g(M_\bullet)\coloneqq c_1 \VV_g(M_\bullet)$ in many instances. From \cref{prop:coherence}\ref{it:coh2}, we can deduce that necessarily $\DD_0(M_\bullet)$ is nef if $V$ is generated in degree one. Starting with the expression \eqref{eq:c1}, it is possible to check positivity properties of $\DD_g(M_\bullet)$ even without knowing whether the sheaf of coinvariants $\VV_g(M_\bullet)$ they depend on is globally generated. In this direction, we point out recent works of Chakravarty \cite{avik} and Choi \cite{choi:discrete}. In particular, \cite{choi:discrete} shows that $-\DD_g(M_\bullet)$ (note the sign!) is nef for every $g$ whenever $V= Vir_{2k+1,2}$, a \textit{discrete series} VOA arising from the Virasoro Lie algebra $\Vir$. This could suggest that, for these types of VOAs, the associated sheaf of conformal blocks is globally generated (in contrast with the case $V=L_\ell(\g)$ where it is the sheaf of coinvariants to be globally generated over $\bM_{0,n})$. 

See also \cite{Fakhruddin,GG:12,giansiracusa:13,AGS:14,BGM:15,mukhopadhyay:ranklevel, BGM:16} for further properties of $\DD_g(M_\bullet)$ in the case $V=L_\ell(\g)$. 

\subsubsection{Modular functor and tensor product} \label{sec:MF} When $V$ is strongly rational, one can also deduce structural properties of the category $\Vmod$ of $V$-modules using conformal blocks. Indeed, in \cite{damiolini.woike} Woike and the author show the following theorem.

\begin{thm} \label{thm:MFS} Coinvariants induce on $\Vmod$ the structure of a modular fusion category. \end{thm}

In particular, this implies that $\Vmod$ has a monoidal structure which satisfies strong compatibilities (e.g. \textit{rigidity}). The tensor product on $\Vmod$ is given by
\begin{equation} \label{eq:cbotimes} M \cbotimes N \coloneqq \bigoplus_{S \in \Ck_V} \VV_0(M, N, S')\otimes_\CC S,
\end{equation} while the associator is induced by an enhanced version of the factorization theorem applied to a 4-pointed sphere.

\pp The main ingredient to show this theorem is the fact that coinvariants for strongly rational VOAs define a \textit{modular functor} \cite{damiolini.woike}. Although this concept is more naturally formulated in a topological framework, from an algebraic geometric perspective, a modular functor can be interpreted as a collection of $\frac{c_V}{2}\Ac_\Lambda$-modules on $\bTMgn$ which satisfy explicit compatibilities with respect to the pullbacks along tautological maps  \cite{BFM, baki}. Oversimplifying, these compatibilities require that the isomorphisms described in \cref{thm:POV} and \cref{thm:sewing} are compatible with action of the Atiyah algebra $\frac{c_V}{2}\Ac_\Lambda$. 

Another important consequence of the fact that spaces of coinvariants define a modular functor is the fact that their dimensions can be computed using a categorical Verlinde formula. We refer to \cite{damiolini.woike,baki, deshpande.mukhopadhyay:2019} for further details.

\begin{rmk}
    \begin{enumerate}[label=(\alph*)]
        \item The fact that coinvariants from VOAs give rise to a modular functor was stated as an expectation by Ben-Zvi and Frenkel in \cite[Section~10.1.4]{bzf}. When $V=L_\ell(\mathfrak{g})$ this was indeed proved in \cite{baki}, who also conjecture that a similar result should hold for a general rational VOA.
        \item The fact that on $\Vmod$ one can define the structure of a modular fusion category has also been shown by Huang and Lepowsky \cite{HL:I,HL:II,HL:III,HL:IV,huang:2005:verlinde,huang:rigidity}, where a direct proof of the Verlinde formula is given. The monoidal structure that they define is given via explicit analytic methods and, although very likely, it is not known whether it coincides with the one obtained via coinvariants. 
        \item  When $V$ is a logarithmic VOA (in particular not rational) \cite{HLZ} show, using analytic methods, that the category $\Vmod$ is a ribbon Grothendieck--Verdier category. In particular, $\Vmod$ is equipped with a monoidal structure which allows one to define the tensor product of two $V$-modules. It would be interesting to understand whether there is an algebro-geometric realization of such a monoidal structure resembling \eqref{eq:cbotimes}.
        \item The tensor category structure on $\Vmod$ described  here has a parallel in the representation theory of quantum groups at  roots of unity \cite{kazhdan.lusztig:I, kazhdan.lusztig:II,  kazhdan.lusztig:III, kazhdan.lusztig:IV, finkelberg,finkelberg:erratum}, where an equivalence with categories of representations of affine Lie algebras  was established; we do not pursue this direction further here.
  \end{enumerate}
\end{rmk}

\subsubsection{Properties of twisted conformal blocks} \label{sec:twCBprop}  We conclude this section summarizing the properties of twisted coinvariants and conformal blocks. For further details, we refer the reader to \cite{frenkelSzczesny:twusted, szczesny:connection, damiolini:2020:conformal,hongkumar:2023,deshpande.mukhopadhyay:2019, hongkumar:2024}. As in \cref{sec:sheaf}, spaces of twisted coinvariants fit together to define a sheaf on $\Mgn^\Gamma$, the moduli space of $n$-pointed $\Gamma$-covers of smooth curves\footnote{As in the untwisted setting, also here one needs to be careful with the dependence of coordinates. We refer to \cite{frenkelSzczesny:twusted,deshpande.mukhopadhyay:2019} for more detail on this matter.} As described in \cite{szczesny:connection}, a $\Gamma$-equivariant version of the Virasoro uniformization equips such sheaves on $\Mgn^\Gamma$ with a projectively flat connection. 

When $V=L_\ell(\g)$, we have already mentioned in \cref{rmk:twCBrmk} that it is possible to extend the notion of twisted coinvariants to allow also admissible $\Gamma$-covers, i.e. $\Gamma$-covers of possibly nodal curves which satisfy prescribed stability conditions (see \cite[Appendix A]{deshpande.mukhopadhyay:2019}). In this situation, most of the known properties of coinvariants extend to the twisted setting, as illustrated in \cref{thm:twCBproperties}. To avoid technicalities, we will assume that $\Gamma$ acts on $\g$ via diagram automorphisms only.

\begin{thm}\label{thm:twCBproperties} 
\begin{enumerate}[label=(\alph*)] 
\item \cite{damiolini:2020:conformal,hongkumar:2023} The sheaf $\VV_g^\Gamma(M_\bullet)$ of twisted coinvariants is a vector bundle of finite rank over $\bMgn^\Gamma$ which is equipped with a projectively flat connection with logarithmic singularities along the boundary $\Delta_{g,n}^\Gamma$ consisting of $\Gamma$-covers of nodal curves. Propagation of vacua \cref{thm:POV}, Factorization \cref{thm:factorization} and Sewing \cref{thm:sewing} hold true \textit{mutatis mutandis} also for twisted coinvariants. 
\item \cite{deshpande.mukhopadhyay:2019} Spaces of twisted coinvariants define a $\Gamma$-twisted modular functor and their dimension is computed with a $\Gamma$-twisted Verlinde formula. \end{enumerate}
\end{thm}

It is an open question to  extend the notion of twisted conformal blocks (from a general VOA) to the whole $\TMgn^\Gamma$ in such a way that, under appropriate assumptions, \cref{thm:twCBproperties} holds true. Furthermore, although the Atiyah algebra acting on $\VV_g^\Gamma(M_\bullet)$ is explicitly computed \cite{deshpande.mukhopadhyay:2019}, a closed formula that computes the Chern classes of $\VV_g^\Gamma(M_\bullet)$ is unknown.

\section{Relation to principal bundles}  \label{sec:BunG}
We now abandon the general case of conformal blocks from VOAs to focus only on the case of (twisted) conformal blocks arising from (twisted) modules for $L_\ell(\g)$. In this situation, in fact, one can show that spaces of (twisted) conformal blocks describe global sections of appropriate line bundles on moduli of principal bundles over a curve \cite{beauville.laszlo:1994:conformal,laszlo.sorger:1997,hongkumar:2023,damiolini.hong.gao:2025}. This geometric realization also holds true for conformal blocks arising from certain lattice VOAs \cite{ueno:CFT}, but it is not know if conformal blocks arising from other VOAs admit such a description.

We begin describing the moduli stack $\Bun_\Gc$, which parametrizes principal $\Gc$-bundle over a smooth projective curve and where $\Gc$ is a parahoric Bruhat--Tits group scheme. In \cref{sec:PicBunG,sec:tCBforBunG} we (conjecturally) describe its Picard group and its relation with (twisted) conformal blocks. Throughout we fix a smooth curve $C$ over $\CC$. 

\subsection{Parahoric Bruhat--Tits group schemes} \label{sec:PBT}  For a reductive group $G$ over $\CC$, we will denote by $\Bun_G$ the stack of principal $G$-bundles over $C$. 

\begin{ex} A good example to keep in mind is that case $G=\SL_r$: in this case $\Bun_{\SL_r}$ can be interpreted as the stack parameterizing vector bundles over $C$ of rank $r$ and with trivial determinant.\end{ex}

In more recent years, more attention has been devoted towards the study of principal bundles under group schemes $\Gc$ defined over the curve $C$ itself and which are not necessarily pullbacks from a group $G$ over $\CC$. We can then consider the stack $\Bun_\Gc$ of $\Gc$-bundles over $C$.

\begin{ex} \begin{enumerate}
    \item For every vector bundle $E$ of rank $r$ over $C$, one can define $\Gc_E$ to be the group of automorphisms of $E$. Since $E$ is locally free, $\Gc_E$ is locally isomorphic to $\GL_r$, but unless $E$ is the trivial vector bundle, then $\Gc_E$ will not be isomorphic to $\GL_r \times C$.
    \item Let $G$ be a reductive group and $B$ a Borel subgroup of $G$. Let us recall that a parabolic $G$-bundle with parabolic structure at points $P_1, \dots, P_n$ is a principal $G$-bundle $\Ec$ over $C$ together with sections $\sigma_i \in \Ec(G/B)(P_i)$. Such bundles can be interpreted as a principal bundle over $C$ under an appropriate group scheme $\Gc_{B,P_\bullet}$ whose generic fiber is $G$, and whose fiber over each $P_i$ is the normal cone of $B$ inside $G$. A similar description occurs choosing, for every point $P_i$, a parabolic subgroup $H_{i}$ of $G$ (which is not necessarily equal to $B$), leading to the group scheme $\Gc_{H_\bullet, P_\bullet}$.
    \item Let $\Gamma$ be a finite group and $\pi \colon X \to C$ be a possibly ramified $\Gamma$-cover. Let $G$ be a smooth group scheme over $\CC$ which is equipped with an action of $\Gamma$. To this datum we can associate the group scheme over $C$ given by the formula
    \[\Gc \coloneqq (\pi_*(G \times X) )^\Gamma.\] This means that for every $C$-scheme $T$, the $T$-points of $\Gc$ are the $\Gamma$ invariant elements of the group $G(T \times_C X)$ (see \cite{neron.models}).  We call such groups \textit{twisted groups}. Note that we can further generalize this example by substituting the constant group scheme $G \times X$ with any smooth group scheme over $X$ equipped with an action of $\Gamma$ lifting the action on $C$.
\end{enumerate}
\end{ex}

\pp Here we will be concerned only with \emph{parahoric Bruhat--Tits} group schemes, which can be seen as a natural generalizations of the examples presented above \cite{heinloth:2010:uniformization, pappas.rapoport:2008:haines}. In order to define what these groups are, we need to set up some notation and we will make some simplifying assumptions. 
If $P$ is a point of $C$, we denote by $\Db_P = \Spec(\Oo_P)$ the formal disk around $P$ and $\Db_P^\times= \Spec(\Kc_P)$ the punctured disk around $P$. By choosing a coordinate $t$ around $P$, we have that $\Oo_P \cong k[\![t]\!]$ and $\Kc_P \cong k(\!(t)\!)$. We will consider the following functors: \begin{itemize}
    \item $\Ls_P\Gc$, given by $\Ls_P\Gc(\Spec(R)) = \Gc(\Db_{P,R}^\times)= \Gc(R(\!(t)\!))$;
        \item $\Ls^+_P\Gc$, given by $\Ls^+_P\Gc(\Spec(R)) = \Gc(\Db_{P,R}) = \Gc(R[\![t]\!])$;
        \item $\Ls_{C\setminus P}\Gc$, given by $\Ls_{C\setminus P} \Gc_P(\Spec(R)) = \Gc((C\setminus P) \times \Spec(R))$.
\end{itemize} Under mild assumptions on $\Gc$, the quotient $\Ls_P\Gc/\Ls^+_P\Gc$ is representable by an ind-scheme called the (twisted) affine flag variety associated with $\Gc$ at $P$ and denoted $\Gr_{\Gc,P}$.

Throughout this note, by a \emph{parahoric Bruhat--Tits group scheme} over $C$ we mean an affine and smooth group scheme $\Gc$ over $C$ such that
\begin{enumerate}[label=(\roman*)]
    \item \label{it:gensimple} $\Gc$ is a generically simply-connected and simple algebraic group;
    \item $\Gc|_P$ is connected for every $P \in C$;
    \item \label{it:parahoric}  $\Gc(\Db_P)$ is a \emph{parahoric subgroup} of $\Gc(\Db_P^\times)$ for every $P \in C$, as in \cite{bruhat.tits:1984:II}.
\end{enumerate}
This last condition can be formulated in various equivalent ways, one of which asserts that the (twisted) affine flag variety $\Gr_{\Gc,P}$ is an ind-proper ind-scheme for every $P \in C$ \cite{richarz:16:indproper}. We refer to \cite{pappas.rapoport:2008:haines,heinloth:2017:stability, damiolini.hong:2023,Pappas-Rapoport:2022} for more details. By condition \ref{it:gensimple}, there are only finitely many ``\textit{bad}'' points where the fiber of $\Gc$ is not reductive, and generically étale the fibers of $\Gc$ are isomorphic to a simple and simply-connected group which we will denote by $G$ throughout. Via the choice of a pinning for $G$, we identify the group of outer automorphism of $G$ with the diagram automorphisms of $\Lie(G) \eqqcolon \g$. 

\begin{rmk} \label{rmk:parahoric} \begin{enumerate} [label=(\alph*)]
\item \label{it:par1} Let $P$ be a parabolic subgroup of $G$. Then one defines the associated parahoric subgroup $H_P$ of $G(\Cl{t})$ to be the preimage of $P$ under the map $G(\CC[\![t]\!]) \to G(\CC)$ which sends $t$ to $0$. The group $H_B$, for $B$ a Borel subgroup of $G$, is called an Iwahori subgroup of $G(\Cl{t})$, whence the term parahoric.
\item \label{it:par2} It is not true, in general, that $\Gc(\DD_P^\times) = G(\Cl{t})$. Instead, since we are working over an algebraically closed field of characteristic zero, one has that $\Gc(\DD_P^\times) \cong (G(\Cl{t^{1/r}}))^\tau$ for a diagram automorphism $\tau$ of $G$ of order $r$. One can identify parahoric subgroups of  $(G(\Cl{t^{1/r}}))^\tau$ with groups of the type $G(\CC[\![t^{1/m}]\!])^\sigma$ for an endomorphism $\sigma$ of $G$ of order $m$ whose outer part coincides with $\tau$.
\item \label{it:par3} Recall that to every parabolic subgroup $P$ of $G$ one can associate a set of vertices $Y_P$ of the Dynkin diagram $I_\g$ of $\g$ (which uniquely identify $P$ up to conjugation). A similar combinatorial description can be carried out for parahoric subgroups as well. In fact, one can associate to every parahoric subgroup $H$ of $(G(\Cl{t^{1/r}}))^\tau$ a nonempty subset, denoted $Y_H$, of the set of vertices of the affine Dynkin diagram $\hat{I}_{\g,r}$ associated to $(G(\Cl{t^{1/r}}))^\tau$. Note that  $\hat{I}_{\g,1}$ is obtained from $I_\g$ by simply adding the special vertex $o$, so that we have $Y_{P} \cup \{ o\} \cong Y_{H_P}$ for every parabolic subgroup $P$ of $G$. In general, every affine Dynkin diagram $\hat{I}_{\g,r}$ has a special vertex denoted $o$. 
\end{enumerate}
\end{rmk}

\subsection{Line bundles on \texorpdfstring{$\Bun_\Gc$}{BunG}} \label{sec:PicBunG} Let $\emptyset \neq S \subset C$ be a non-empty set of $C$. 
In view of the uniformization theorem \cite{beauville.laszlo:1994:conformal,heinloth:2010:uniformization}, the stack $\Bun_\Gc$ has the following quotient presentation
\begin{equation}\label{eq:uniformization} \Bun_\Gc \cong \Ls_{C\setminus S}\Gc \setminus \Gr_{\Gc,S} = \Ls_{C\setminus S}\Gc \setminus \left( \prod_{P \in S} \Gr_{\Gc,P}\right).\end{equation} This has two important consequences which we collect in the following remark.

\begin{rmk} \label{rmk:consequences-uniformization}
\begin{enumerate}[label=(\alph*)]
\item \label{it:cons1} A line bundle on $\Bun_\Gc$ is uniquely described by the datum of a line bundle on $\Gr_{\Gc,S}$ together with a $\Ls_{C\setminus S}\Gc$-linearization. When $\Gc_{C \setminus S}$ is ind-integral and every fiber of  the restriction of $\Gc$ to $C \setminus S$ is simple, then one can prove that if such a linearization exists, then it must be unique.
\item \label{it:cons2} For every line bundle $\Lc$ on $\Bun_\Gc$, one has the identification
\[ \Ho^0(\Bun_\Gc, \Lc) = \Ho^0(\Gr_{\Gc,S}, q^*\Lc)^{\Ls_{C\setminus S}\Gc},\] where $q \colon \Gr_{\Gc,S} \to \Bun_\Gc$ denotes the quotient map inducing the isomorphism \eqref{eq:uniformization}. Note that, when $\Ls_{C\setminus S}\Gc$ is ind-integral, then taking global sections of $q^*\Lc$ which are invariant under $\Ls_{C\setminus S}\Gc$, is the same as taking global sections of $q^*\Lc$ on which the Lie algebra $\Lie(\Gc(C\setminus S))$ acts trivially.
\end{enumerate}\end{rmk}

In both \ref{it:cons1} and \ref{it:cons2} above, a crucial property that one would want is that the ind-group scheme $\Gc(C\setminus S)$ is ind-integral. Up to enlarging $S$, if necessary, this is indeed true. Namely:

\begin{prop} \label{prop:ind-integral} [$\Gc$ constant \cite{beauville.laszlo:1994:conformal}, $\Gc$ twisted \cite{hongkumar:2023}, $\Gc$ parahoric \cite{damiolini.hong.gao:2025}] 
    If $S$ contains all bad points of $\Gc$, then $\Gc(C\setminus S)$ is ind-integral.
\end{prop}

One can show that, under the assumptions of \cref{prop:ind-integral}, the pullback map $q^* \colon \Pic(\Bun_\Gc) \to \Pic(\Gr_{\Gc,S})$ is injective. We are left to understand the image of $q^*$. To do so, it will be helpful to describe $\Pic(\Gr_{\Gc,S})$ in a more combinatorial way. In view of \cite{pappas.rapoport:2008:haines,zhu:2014:coherence}, one has a natural description \[\Pic(\Gr_{\Gc,S}) =\prod_{P \in S} \Pic(\Gr_{\Gc,P}) \cong \prod_{P \in S} \bigoplus_{i \in Y_{\Gc(\DD_P)}} \Lambda_i \ZZ\] where $\Lambda_i$ denotes the $i$-th fundamental weight of the affine Dynkin diagram associated with $\Gc(\DD_P^\times)$. This expression allows one to define the \emph{central charge} map
\begin{equation} \label{eq:ccGr}
    \cc_P \colon \Pic(\Gr_{\Gc,P}) \to \ZZ, \qquad  \cc \left(\sum_{i \in Y_{\Gc(\DD_P)}} n_i \Lambda_i \right) \coloneqq \sum_{i \in Y_{\Gc(\DD_P)}} n_i \check{a}_i,
\end{equation}
where $\check{a}_i \in \{1, \dots, 6\}$ are the \textit{dual Kac labels} (see \cite[\S6.1]{kac:1990:infinite}). As $P$ varies in $S$ the assignment \eqref{eq:ccGr} defines the map
\[ \cc_S \colon \Pic(\Gr_{\Gc,S}) \longrightarrow \bigoplus_{P \in S}\ZZ.
\] Let $\Delta \colon \ZZ \to  \bigoplus_{x \in S}\ZZ$ be the diagonal inclusion, and set $\Pic^\Delta(\Gr_{\Gc,S}) \coloneqq \cc_S^{-1}(\Delta(\ZZ))$, that is $\Pic^\Delta(\Gr_{\Gc,S})$ is the group of line bundles on $\Gr_{\Gc,S}$ whose central charge, computed using the various points $P \in S$, is constant. 

\begin{prop} \cite[Proposition 4.1]{zhu:2014:coherence} Assume that $S$ contains all the bad points of $\Gc$. Then the image of $q^*$ is contained in $\Pic^\Delta(\Gr_{\Gc,S})$. In particular we have a well defined notion of central charge $\cc$ for every line bundle of $\Bun_\Gc$. \end{prop}

In \cite{damiolini.hong.gao:2025} Hong and the author conjectured that indeed $q^*$ provides an isomorphism between $\Pic(\Bun_\Gc)$ and $\Pic^\Delta(\Gr_{\Gc,S})$.\footnote{In a very recent pre-print \cite{balaji.pandey}, the authors show that this is indeed the case. I am not able to understand their argument and therefore to confidently assert that this conjecture has been solved.}

\begin{rmk}
    In \cite{heinloth:2010:uniformization} it is shown that there is an exact sequence
    \[0 \to \prod_{P \in S} \textrm{Char}(\Gc_P) \longrightarrow \Pic(\Bun_\Gc) \overset{\cc}{\longrightarrow} \cc_\Gc \ZZ \to 0\] for some constant $\cc_\Gc$. Therefore the equality of $\Pic(\Bun_\Gc)$ with $\Pic^\Delta(\Gr_{\Gc,S})$ is numerically equivalent to \[\cc_\Gc = \underset{P \in S}{\operatorname{lcm}}\left(\underset{i \in Y_{\Gc(\DD_P))}}{\operatorname{gcd}}{\check{a}_i} \right).\]

\end{rmk}

\subsection{Sections of line bundles and (twisted) conformal blocks}  \label{sec:tCBforBunG} We describe here the relation between sections of line bundles over $\Bun_\Gc$ and twisted conformal blocks. As an application, we describe a criterion to detect descent from $\Gr_{\Gc,S}$ to $\Bun_\Gc$. A different approach to this problem was given by Faltings in \cite{faltings:loop}. Throughout this section $S$ will denote a non-empty subset of $C$ containing all the bad points of $\Gc$.

\pp Consider the line bundle $\Lc_\Lambda$ on $\Gr_{\Gc,P}$ associated with $\Lambda = \sum n_i \Lambda_i$. Assume that $\Lambda$ is dominant, i.e. that $\Lambda \neq 0$ and $n_i \geq 0$ for every $i$. As shown in \cite{faltings:verlinde,hongkumar:2023}, the representation $\Hc(\Lambda)$ of the central extension of the Lie algebra $\Lie(\Gc(\DD_P^\times))$ can be integrated to a projective representation of ${\Ls_P\Gc}(\Lambda)$.

By applying a generalization of the Borel--Weil theorem to $\Gr_{\Gc,P}$  one obtains the identification
\begin{equation} \label{eq:H0GrG} \Ho(\Gr_{\Gc,P}, \Lc_{\Lambda}) = \Hc(\Lambda)^\vee = \Hom_\CC(\Hc(\Lambda), \CC).
\end{equation}
We observe that the identification \eqref{eq:H0GrG} is of $\widehat{\Ls_P\Gc}(\Lambda)$-modules, where $\widehat{\Ls_P\Gc}(\Lambda)$ is the  central extension of $\Ls_P\Gc$ determined by the projective action of ${\Ls_P\Gc}(\Lambda)$ on $\Hc(\Lambda)$.

\begin{ex} \label{ex:Lambda}
   \begin{enumerate} [label=(\alph*)] \item  When $\Gc=G$, then one has that $\Pic(\Gr_{G,P}) = \ZZ \Lambda_o$, where $o$ is the special vertex of the Dynkin diagram of $\gh$. The representation $\Hc(\ell \Lambda_o)$ coincides with the representation denoted $L_{0,\ell}$ in \cref{sec:CBfromLieAlgebras}. The Lie algebra of $\widehat{\Ls_PG}(\ell\Lambda_o)$ coincides with $\gh$, where the center acts by $\ell$. 
    
    \item Consider the parabolic case, i.e. assume that $\Gc(\DD_P)=H_B$ for a $B \subset G$ a Borel subgroup. Recall the identification $Y_{H_B}= Y_P \cup \{o\} = I_g \cup \{0\}$ described in \cref{rmk:parahoric} \ref{it:par3}. This identification can be interpreted as a way to decompose every non-special weight $\Lambda_i$  (i.e. $i \neq o$) as a sum $\lambda_i + \check{a}_i \Lambda_o$ where $\lambda_i$ are weights of the  finite dimensional Lie algebra $\g$ and  $\check{a}_i = \cc(\Lambda_i)$ as in \eqref{eq:ccGr}. It follows that, if $\Lambda = \sum n_i \Lambda_i$, then the $\gh$ representation $\Hc(\sum n_i \Lambda_i)$ equals $L_{\ell, \lambda}$ for $\ell = \sum n_i\check{a}_i$  and $\lambda = \sum_{i \neq o} n_i \lambda_i$.  

    \item For a general parahoric subgroup of $(G(\Cl{t^{1/r}}))^\tau$, we have that $\Hc(\Lambda)$ is an irreducible integrable representation of the twisted affine Lie algebra $\gh^\tau=(\g(\!(t)\!) \oplus \CC)^\tau$. In the language of twisted modules for VOAs, these modules are exactly the $\tau$-twisted $L_{\ell}(\g)$-modules (see \cref{sec:twCB}). \end{enumerate}
\end{ex}

Consider now $\Lc_{\vec{\Lambda}}$, the line bundle on $\Gr_{\Gc,S}$ corresponding to the tuple $\vec{\Lambda} \in \Pic^\Delta(\Gr_{\Gc,S})$ for dominant $\Lambda$s. Combining  \cref{rmk:consequences-uniformization} (2) with \eqref{eq:H0GrG} we obtain that, if $\Lc$ is the line bundle on $\Bun_\Gc$ pulled back from $\Lc_{\vec{\Lambda}}$---in formulas if $q^*\Lc = \Lc_{\vec{\Lambda}}$---then
\[ \Ho^0(\Bun_\Gc, \Lc) = \Ho^0(\Gr_{\Gc,S}, \Lc_{\vec{\Lambda}})^{\Ls_{C\setminus S}\Gc} =  (\Hc(\vec{\Lambda})^\vee)^{\Ls_{C\setminus S}\Gc}.\]
Since $\Ls_{C\setminus S}\Gc$ is ind-integral, then we also have
\[ \Ho^0(\Bun_\Gc, \Lc) = (\Hc(\vec{\Lambda})^\vee)^{\Lie(\Gc(C\setminus S))}= \Hom_{\Lie(\Gc(C\setminus S))}(\Hc(\vec{\Lambda}), \CC),\] and there is a unique way in which the Lie algebra $\Lie(\Gc(C\setminus S))$ acts on $\Hc(\vec{\Lambda})$. 

As noted in \cref{ex:Lambda}, when $\Gc$ is parabolic (or constant), the datum $\vec{\Lambda}$ can be identified with $(\vec{\lambda}, \ell)$ for $\lambda_i \in P^+_\ell$ and $\ell=\cc(\vec{\Lambda})$.  It follows that, when $\Gc$ is parabolic (or constant), one obtains an identification
\begin{equation}  \label{eq:H0=CB} \Ho^0(\Bun_\Gc, \Lc) \cong \VV(C; P_\bullet; \Hc(\vec{\Lambda}))^\dagger,\end{equation} 
where, on the right hand side we wrote $\Hc(\vec{\Lambda})$ in place of the notation $\Hc_{\Lambda_\bullet}$ used in \cref{sec:CB}. The isomorphism \eqref{eq:H0=CB} has been shown in \cite{beauville.laszlo:1994:conformal, laszlo.sorger:1997,sorger:99:moduli}.

\pp An isomorphism analogous to \eqref{eq:H0=CB} also holds true for every parahoric group $\Gc$ and it involves twisted conformal blocks (even if $\Gc$ is not a twisted group). To state this result, we first need to set up some notation. Assume that $S \subset C$ is a non-empty and finite set containing all bad points of $\Gc$. Then in \cite{damiolini.hong:2023} it is shown that there exists a subgroup $\Gamma$ of the diagram automorphisms of $\g$ and an étale $\Gamma$-cover 
\begin{equation*} \label{eq:twG} \pi \colon X \to C\setminus S  \quad \text{ such that } \quad \Gc|_{C\setminus S} \cong \pi_*(G \times X)^\Gamma.\end{equation*}
The cover $\pi$ naturally extends to a possibly ramified $\Gamma$-cover $\overline{\pi} \colon \overline{X} \to C$. Let us choose, for every $P \in S$, a point $x \in \overline{X}$ such that $\overline{\pi}(x)=P$.  From \cref{ex:Lambda}, we deduce that the datum of a dominant $\Lambda \in \Pic(\Gr_{\Gc,P})$ is equivalent to the datum of a $\tau_x$-twisted $L_{\cc(\Lambda)}(\g)$-module, where $\tau_x$ is the stabilizer of the point $x \in \overline{X}$ above $P$. We can then apply the same steps used to show \eqref{eq:H0=CB} to show that \eqref{eq:H0=CB} extends beyond the parabolic case.

\begin{thm} \cite{damiolini.hong.gao:2025} If a dominant line bundle $\Lc(\vec{\Lambda})$ descends to a line bundle $\Lc$ on $\Bun_\Gc$, there is a natural isomorphism
\begin{equation} \label{eq:H0=twCB}
    \Ho^0(\Bun_\Gc, \Lc) \cong \VV(\overline{\pi} \colon \overline{X} \to C; x_\bullet, \Hc(\vec{\Lambda}))^\dagger.
\end{equation}
\end{thm}

A consequence of this result is that the space of sections is finite dimensional and, using \cite{deshpande.mukhopadhyay:2019}, it is computable via the twisted Verlinde formula. We refer to \cite{nick} for a generalization  of \eqref{eq:H0=CB} involving higher cohomology on the left-hand side and chiral homology on the right-hand side. The isomorphism \eqref{eq:H0=CB} also plays an important role in the study of strange duality maps between moduli of principal bundles \cite{belkale:strange, 
pauly:SD:asterisque, mukhopadhyay.wentworth:SD, martens.etal:2412}.  It would be interesting to know if \eqref{eq:H0=twCB} can have 
similar generalizations and applications. 
Finally, we note that under the identification \eqref{eq:H0=CB}, the connection of \cref{prop:connectionMgn} coincides with the Hitchin connection on the bundle of generalized theta functions \cite{laszlo:connection}, a result that has been generalized to the parabolic setting in \cite{biswas.mukhopadhyay.wentworth:hitchin} and to twisted groups in \cite{geometrization.tuy}.

\pp We conclude describing how non-vanishing of twisted conformal blocks provides a strategy to control descent from $\Gr_{\Gc,S}$ to $\Bun_\Gc$, namely: 

\begin{prop}\cite{sorger:99:moduli,damiolini.hong.gao:2025}  \label{prop:criterion}
 If $\VV(\vec{\Lambda})\coloneqq \VV(\overline{\pi} \colon \overline{X} \to C; x_\bullet, \Hc(\vec{\Lambda}))$ is non zero, then $\Lc_{\vec{\Lambda}}$ descends to a (necessarily unique) line bundle on $\Bun_\Gc$.
\end{prop}

We give a sketch of the proof. 

\begin{proof} One first shows that  $\Ls_S\Gc \coloneqq \prod_{P \in S} \Ls_P\Gc$ acts projectively on $\Hc(\vec{\Lambda})$ in such a way that $\Ls_{C\setminus S}\Gc$ has an induced projective action $\rho$ on the spaces of coinvariants $\VV(\vec{\Lambda})$, which is non zero by assumptions \cite{faltings:verlinde,hongkumar:2023}. Denoting by $\widehat{\Ls_{C\setminus S}\Gc}(\vec{\Lambda})$ the induced central extension of $\Ls_{C\setminus S}\Gc$, we obtain the diagram
\begin{equation} \label{eq:LGh} \xymatrix{
1  \ar[r] & \Gm \ar[r]\ar[d] & \widehat{\Ls_{C\setminus S}\Gc}(\vec{\Lambda}) \ar[r] \ar[d] & \Ls_{C\setminus S}\Gc \ar[r] \ar[d]^{\rho} & 1\\
1 \ar[r] & \Gm \ar[r] & \mathrm{GL}(\VV(\vec{\Lambda})) \ar[r] & \mathrm{PGL}(\VV(\vec{\Lambda})) \ar[r] & 1 }\end{equation}
and note that the line bundle $\Lc_{\vec{\Lambda}}$ is detected by a (necessarily unique) splitting of (the first row of) \eqref{eq:LGh}. Observe now that the derivative of $\rho$ is necessarily zero by definition of spaces of coinvariants. But now, since $\Ls_{C\setminus S}\Gc$ is ind-integral (in this section we assume that $S$ contains all the bad points, so we can use \cref{prop:ind-integral}), this implies that $\rho$ sends the whole group $\Ls_{C\setminus S}\Gc$ to the identity elements $\mathrm{Id}_{\VV(\vec{\Lambda})}$, providing in this way the wanted splitting.   
\end{proof}

Using this method---and making use of the properties of twisted conformal blocks from \cref{sec:twCBprop}---it is shown in \cite{damiolini.hong.gao:2025} that indeed $q^*\Pic(\Bun_\Gc)=\Pic^\Delta(\Gr_{\Gc,S})$ whenever the special vertex $o$ belongs to $Y_{\Gc(\DD_P^\times)}$ for every $P \in S$.

\bibliographystyle{alpha}
\bibliography{biblio}

\end{document}